\documentclass[a4paper, 11pt]{amsart}

\usepackage{cite}
\usepackage[dvipdfmx]{graphicx}
\usepackage[dvipdfmx]{color}
\usepackage{amsmath}
\usepackage{amssymb}
\usepackage{amsthm}
\usepackage{bm}
\usepackage{eucal}
\usepackage{braket}
\usepackage{stmaryrd}
\usepackage{mathrsfs}
\calclayout

\newcommand{\pr}[1]{#1^{\prime}}

\newcommand{\del}{\partial}

\newcommand{\dist}{\mathrm{dist}}
\newcommand{\supp}{\mathrm{supp}}
\newcommand{\rmIm}{\mathrm{Im}}

\newcommand{\wtilde}[1]{\widetilde{#1}}

\newcommand{\bR}{\mathbb{R}}
\newcommand{\bN}{\mathbb{N}}
\newcommand{\bH}{\mathbb{H}}
\newcommand{\bC}{\mathbb{C}}
\newcommand{\bP}{\mathbb{P}}
\newcommand{\bE}{\mathbb{E}}
\newcommand{\cF}{\mathcal{F}}
\newcommand{\GFF}{\mathrm{GFF}}
\newcommand{\law}{\mathrm{(law)}}
\newcommand{\ul}[1]{\underline{#1}}
\newcommand{\ol}[1]{\overline{#1}}
\newcommand{\bsl}{\backslash}
\newcommand{\hcap}{\mathrm{hcap}}
\newcommand{\fh}{\mathfrak{h}}

\newcommand{\cC}{\mathcal{C}}

\newcommand{\cU}{\mathcal{U}}

\theoremstyle{plain}
\newtheorem{thm}{Theorem}[section]
\newtheorem{lem}[thm]{Lemma}
\newtheorem{prop}[thm]{Proposition}

\theoremstyle{definition}
\newtheorem{defn}[thm]{Definition}%[section]
\theoremstyle{remark}
\newtheorem{rem}[thm]{Remark}%[section]

\makeatletter
    
    \@addtoreset{equation}{section}
\makeatother

\title[]{Three phases of multiple SLE driven by non-colliding Dyson's Brownian motions}
\date{\today}
\author{Makoto Katori}
\address[Makoto KATORI]{Department of Physics, Faculty of Science and Engineering, Chuo University, Kasuga, Bunkyo-ku, Tokyo 112-8551, Japan}
\email{katori@phys.chuo-u.ac.jp}

\author{Shinji Koshida}
\address[Shinji KOSHIDA]{Department of Physics, Faculty of Science and Engineering, Chuo University, Kasuga, Bunkyo-ku, Tokyo 112-8551, Japan. Department of Mathematics and Systems Analysis, Aalto University, 00076 Espoo, Finland}
\email{koshida@phys.chuo-u.ac.jp shinji.koshida@aalto.fi}

\begin{document}

\begin{abstract}
The present paper is concerned with properties of multiple Schramm--Loewner evolutions (SLEs) labelled by a parameter $\kappa\in (0,8]$.
Specifically, we consider the solution of the multiple Loewner equation driven by a time change of Dyson's Brownian motions in the non-colliding regime.
Although it is often considered that several properties of the solution can be studied by means of commutation relations of SLEs and the absolute continuity,
this method is available only in the case that the curves generated by commuting SLEs are separated.
Beyond this restriction, it is not even obvious that the solution of the multiple Loewner equation generates multiple curves.
To overcome this difficulty, we employ the coupling of Gaussian free fields and multiple SLEs.
Consequently, we prove the longstanding conjecture that the solution indeed generates multiple continuous curves.
Furthermore, these multiple curves are (i) simple disjoint curves when $\kappa\in (0,4]$, (ii) intersecting curves when $\kappa\in (4,8)$, and (iii) space-filling curves when $\kappa=8$.
\end{abstract}


\maketitle

\section{Introduction}
\subsection{Backgrounds}
Since its introduction in \cite{Schramm2000}, Schramm--Loewner evolution (SLE) has been playing one of the central roles in studying critical systems in two dimensions. It is a theory of a random continuous curve as a stochastic analogue of the classical Loewner theory, which we briefly recall. Let $\bH$ be the complex upper half plane. We consider the chordal Loewner equation driven by a continuous function $\xi \colon[0,\infty)\to\bR$:
\begin{equation}
\label{eq:Loewner}
	\frac{d}{dt}g_{t}(z)=\frac{2}{g_{t}(z)-\xi (t)},\quad t\geq 0, \quad g_{0}(z)=z\in\bH.
\end{equation}
For each $z\in \bH$, we define
\begin{equation*}
	\tau_{z}:=\sup\set{t\ge 0: \mbox{The solution of (\ref{eq:Loewner}) is well defined up to $t$},\ g_{t}(z)\in\bH},
\end{equation*}
and, for each $t\geq 0$, we set $K_{t}:=\set{z\in\bH|\tau_{z}\leq t}$. Then, $g_{t}$ is the conformal map from $\bH\bsl K_{t}$ to $\bH$ under the hydrodynamical normalization: $\lim_{z\to\infty}|g_{t}(z)-z|=0$.

A particularly significant example is the case when the driving function is given by $\xi (t)=B_{\kappa t}$, $t\geq 0$, where $(B_{t}:t\geq 0)$ is a standard Brownian motion and $\kappa>0$. In this case, almost surely, the limit $\eta (t):=\lim_{\varepsilon\searrow 0}g_{t}^{-1}(\xi (t)+\sqrt{-1}\varepsilon)\in\ol{\bH}$ exists for all $t\geq 0$, $\eta \colon [0,\infty)\to\ol{\bH}$ is a continuous curve called an SLE-curve such that $\lim_{t\to\infty}\eta (t)=\infty$, and at each $t\geq 0$, $\bH\bsl K_{t}$ is the unbounded component of $\bH\bsl \eta (0,t]$. In particular, $\eta$ defines a probability law of a curve in $\ol{\bH}$ connecting $0$ and $\infty$. We call the probability law or the family of conformal maps $(g_{t}:t\geq 0)$ specified by the parameter $\kappa$ the SLE$(\kappa)$. It was argued in \cite{Schramm2000} that SLE$(\kappa)$ with varying $\kappa>0$ covers the probability laws on a curve in $\ol{\bH}$ connecting $0$ and $\infty$ that exhibit the {\it conformal invariance} and the {\it domain Markov property}.

The SLE-curves fall into three phases depending on $\kappa$ \cite{LawlerSchrammWerner2004,RohdeSchramm2005}. If $\kappa\in (0,4]$, $\eta$ is almost surely a simple curve. If $\kappa\in (4,8)$, $\eta$ is almost surely self-intersecting and hits $\bR$. If $\kappa>8$, $\eta$ is almost surely space-filling.

In relation to critical systems, an SLE-curve is typically a candidate for a scaling limit of a cluster interface in a critical lattice system and, in some models including the critical percolation \cite{Smirnov2001}, the uniform spanning tree and the related loop-erased random walk \cite{LawlerSchrammWerner2004}, the critical Ising model \cite{ChelkakDuminil-CopinHonglerKemppainenSmirnov2014}, this property has been proved. In this context of critical systems, it is natural to extend the notion of SLE so to describe random multiple curves, which is where the notion of multiple SLE takes place. There are several formulations of multiple SLE, especially two {\it local} formulations; one way is to define a multiple SLE as a collection of commuting Loewner chains \cite{Dubedat2007}, and the other way is to consider the multiple Loewner equation
\begin{align}
\label{eq:multiple_Loewner}
	\frac{d}{dt}g_{t}(z)&=\sum_{i=1}^{N}\frac{2}{g_{t}(z)-X_{t}^{(i)}},\quad t\ge0, & g_{0}(z)&=z\in\bH
\end{align}
driven by $N$ continuous stochastic processes $\left(X^{(i)}_{t}:t\geq 0\right)$, $i=1,\dots, N$ as was first introduced in \cite{BauerBernardKytola2005}.
(Yet another formulation of multiple SLE \cite{KozdronLawler2007, Lawler2009c, BeffaraPeltolaWu2018} is referred to as {\it global}.)
These days, the definition as commuting Loewner chains has been more often adopted in literatures (e.g. in \cite{PeltolaWu2019,Karrila2019, Karrila2020, Izyurov2020}). In fact, the definition is sufficient as far as the probability law of multiple curves is concerned, and more convenient to analyze since each curve can be studied separately. On the other hand, we can only count few papers \cite{Graham2007, RothSchleissinger2017, HottaKatori2018} where the multiple Loewner equation is discussed, and little is known about basic properties of its solution.
The study of the multiple Loewner equation was initiated by Cardy in \cite{Cardy2003a} in the {\it radial} case
under the motivation of understanding multiple cluster interfaces.
An advantage of the multiple Loewner equation is that it is more natural in the context of complex analysis and makes clearer sense as a dynamical system.
Furthermore, the multiple Loewner equation manifests the potential interrelation between cluster interfaces and many particle systems
as was already pointed out in \cite{Cardy2003a}.
These aspects of the multiple Loewner equation motivate us to study it in the present paper.

Before formulating our problem, let us see a general property of the solution to the multiple Loewner equation (\ref{eq:multiple_Loewner}).
For each $z\in\bH$, we define a stopping time
\begin{equation*}
	\tau_{z}:=\sup\set{t\ge 0: \mbox{The solution of (\ref{eq:multiple_Loewner}) is well defined up to $t$},\ g_{t}(z)\in\bH},
\end{equation*}
and for each $t\ge 0$, we set $K_{t}:=\set{z\in\bH:\tau_{z}\leq t}$. Then at each $t\ge 0$, $g_{t}$ is the hydrodynamically normalized conformal map from $\bH\bsl K_{t}$ to $\bH$. 

\subsection{Problem setting}
Let us fix $N\in\bN$, $\beta\geq 0$. The $N$-particle Dyson's Brownian motion model of parameter $\beta$ is the system of stochastic differential equations (SDEs) on $\left(Y^{\beta;(i)}_{t}:t\geq 0\right)$, $i=1,\dots, N$ that reads \cite{Dyson1962}
 \begin{equation}
 \label{eq:Dyson}
 	dY^{\beta;(i)}_{t}=dB^{(i)}_{t}+\frac{\beta}{2}\sum_{\substack{j=1 \\ j\neq i}}^{N}\frac{dt}{Y^{\beta;(i)}_{t}-Y^{\beta;(j)}_{t}}, \quad t\geq 0, \quad i=1,\dots, N,
 \end{equation}
 where $(B^{(i)}_{t}:t\ge 0)$, $i=1,\dots, N$ are independent standard Brownian motions.
 It is known \cite{CepaLepingle1997} that the Dyson's Brownian motion model (\ref{eq:Dyson}) have a unique strong continuous solution for arbitrary initial conditions $Y^{\beta;(i)}_{0}=Y_{i}\in\bR$, $i=1,\dots, N$, and that the solution is almost surely non-colliding if $\beta\geq 1$.
We call a solution of the Dyson's Brownian motion model (\ref{eq:Dyson}) a system of $N$ Dyson's Brownian motions of parameter $\beta$, or simply, Dyson's Brownian motions.

In this paper, we study the solution of the multiple Loewner equation (\ref{eq:multiple_Loewner}) when the driving processes are given by $X^{(i)}_{t}=Y^{8/\kappa;(i)}_{\kappa t}$, $t\geq 0$, $i=1,\dots, N$ with $\kappa\in (0,8]$ such that the initial points $X_{i}:=X^{(i)}_{0}$, $i=1,\dots, N$ are distinct. Specifically, these driving processes are almost surely non-colliding, hence, they make sense at any $t\geq 0$. Note that, in general, Dyson's Brownian motions can be started from an identical point and multiple SLE driven by such Dyson's Brownian motions was of importance in \cite{HottaKatori2018}. In the current paper, however, we concentrate on the case when the initial conditions are given by distinct points.  We call the solution $(g_{t}:t\geq 0)$ the $N$-SLE$(\kappa)$ driven by Dyson's Brownian motions, or the $N$-SLE$(\kappa)$ for short, starting at $\bm{X}=(X_{1},\dots, X_{N})$. Let us notice, to be sure, that the driving processes satisfy the system of SDEs
\begin{align}
\label{eq:Dyson_timechanged}
	dX_{t}^{(i)}=\sqrt{\kappa}dB_{t}^{(i)}+\sum_{\substack{j=1\\ j\neq i}}^{N}\frac{4dt}{X^{(i)}_{t}-X^{(j)}_{t}},\quad t\ge 0, \quad  i=1,\dots, N.
\end{align}

\begin{rem}
It is natural to perform a time change of Dyson's Brownian motions to obtain driving processes as an analogy to the case of the ordinary SLE. On the other hand, it is more non-trivial that the strength of the interaction is fixed at $\beta=8/\kappa$. This phenomenon has been repeatedly observed in literatures \cite{Cardy2003a, BauerBernardKytola2005, Dubedat2007} and is understood as a consequence of the fact that the multiple SLE is consistent with a critical system. In fact, we will rely on coupling of multiple SLE and Gaussian free field (GFF), a representing example of critical system.
\end{rem}

\begin{rem}
In \cite{BauerBernardKytola2005,Dubedat2007}, the authors studied conditions on a multiple SLE so that it describes multiple cluster interfaces of a critical system.
In consequence, they found that multiple driving processes must be determined by a single {\it partition function}.
From this perspective, the system of SDEs (\ref{eq:Dyson_timechanged}) should be expressed as~\cite{BauerBernardKytola2005}
\begin{align*}
	dX_{t}^{(i)}=&\sqrt{\kappa}dB_{t}^{(i)}+\kappa\left(\del_{x_{i}}\log Z\right)\left(X^{(1)}_{t},\dots, X^{(N)}_{t}\right)dt+\sum_{\substack{j=1\\ j\neq i}}^{N}\frac{2dt}{X^{(i)}_{t}-X^{(j)}_{t}},\\
	&t\ge 0, \quad  i=1,\dots, N,
\end{align*}
where $Z=Z(x_{1},\dots, x_{N})$ is the partition function given by
\begin{equation*}
	Z(x_{1},\dots, x_{N})=\prod_{1\leq i<j\leq N}|x_{i}-x_{j}|^{2/\kappa},\quad x_{i}\in \mathbb{R},\; i=1,\dots, N,\quad x_{i}\neq x_{j},\; i\neq j.
\end{equation*}
\end{rem}

The goal of the present paper is to prove the following theorem:
\begin{thm}
\label{thm:main}
For each $i=1,\dots, N$, almost surely, the limit $\eta^{(i)}(t)=\lim_{\varepsilon\downarrow 0}g^{-1}_{t}(X^{(i)}_{t}+\sqrt{-1}\varepsilon)$ exists for all $t\geq 0$ and $\lim_{t\to\infty}|\eta^{(i)}(t)|=\infty$. Furthermore,
\begin{enumerate}
\item 	if $\kappa\in (0,4]$, $\eta^{(i)}$, $i=1,\dots, N$ are almost surely simple disjoint curves such that $\eta^{(i)}(0,\infty)\subset\bH$, $i=1,\dots, N$,
\item 	if $\kappa\in (4,8)$, $\eta^{(i)}$, $i=1,\dots, N$ are almost surely continuous curves that intersect themselves and $\bR$,
\item 	if $\kappa =8$, $\eta^{(i)}$, $i=1,\dots, N$ are almost surely space filling continuous curves.
\end{enumerate}
\end{thm}

\begin{rem}
\begin{enumerate}
\item 	Theorem \ref{thm:main} implies that, at each $t\geq 0$, the domain $\bH\bsl K_{t}$ is the unbounded component of $\bH\bsl\bigcup_{i=1}^{N}\eta^{(i)}(0,t]$.
\item 	Even if $\kappa>4$, two curves do not cross each other. Indeed, crossing of curves contradicts the continuity of the driving processes.
\item 	If $\kappa>4$, each curve $\eta^{(i)}$, $i=1,\dots, N$ hits the next curves, but does not hit the second next curves. In fact, if it hits a second next curve, then it contradicts the fact that the next curve in between goes to infinity.
\end{enumerate}
\end{rem}

\subsection{Strategy}
Before illustrating our strategy to prove Theorem~\ref{thm:main}, let us consider an alternative one that can be employed when $\kappa\in (0,4]$.
As we have already mentioned, there are two distinguished (local) formulations of multiple SLE; the commuting family of Loewner chains \cite{Dubedat2007} and the multiple Loewner evolution \cite{BauerBernardKytola2005}.
A possible strategy is to prove the equivalence between these two formulations.
In fact, they can be shown to be equivalent at the infinitesimal level by means of the generators of the relevant Markov processes though an explicit proof has not appeared to the authors' knowledge.
It is well-known that each member of a commuting Loewner chain of parameter $\kappa$ is absolutely continuous with respect to the SLE$(\kappa)$.
Therefore, what happens almost surely to SLE$(\kappa)$ happens almost surely to the members of the commuting Loewner chains.
In consequence, we can study the random geometry generated by the multiple SLE regardless of how it is formulated.

The significant assumption in the above argument is that the commuting family of Loewner chains defines multiple curves, or in other words,
the infinitesimal commutation relations among Loewner chains are enhanced to the finite ones.
Note that, however, infinitesimal commutation relations can be integrated to finite ones under the assumption that the curves generated by the Loewner chains are separated, which is valid only when $\kappa\in (0,4]$.
Only in such a case, the infinitesimal equivalence of the two formulations of multiple SLE is also integrated.

We hope to go further beyond this restriction.
Indeed, under the assumption that curves are separated, it is in principle impossible to study if two curves touch each other or not, as is stated in Theorem~\ref{thm:main}.
For that purpose, we have to employ another method to ensure the existence of multiple curves.
Below, we describe our strategy to prove Theorem~\ref{thm:main}.

We rely on the notion of coupling between multiple SLE and GFF, which has been studied in \cite{Dubedat2009,MillerSheffield2016a}.
The relevant GFF is given by
\begin{equation}
\label{eq:relevant_GFF}
	h=H-\frac{2}{\sqrt{\kappa}}\sum_{i=1}^{N}\arg (\cdot-X_{i}),
\end{equation}
where $H$ is the zero-boundary GFF on $\bH$, the function $\arg$ takes values in $(0,\pi)$, and $X_{i}\in\bR$, $i=1,\dots, N$, which are supposed to give initial conditions of Dyson's Brownian motions, are distinct.
Note that the GFF (\ref{eq:relevant_GFF}) is the only possibility that can be coupled with multiple SLE when $\kappa\neq 4$ \cite{Koshida2019}.
In \cite{KangMakarov2013}, it is understood that the GFF (\ref{eq:relevant_GFF}) is a combination of a central charge deformation of the zero-boundary GFF and an insertion of primary fields in the language of conformal field theory.

It is worth noting that the coupling between SLE and GFF has been studied in many other different settings; different domains and different boundary conditions in e.g.~\cite{IzyurovKytola2013,QianWerner2018, ByunKangTak2018}.

For each $i=1,\dots, N$, a random continuous curve $\gamma^{(i)}:[0,\infty)\to\ol{\bH}$ such that $\gamma^{(i)}(0)=X_{i}$ is characterized by the coupling property with the GFF $h$. Furthermore, under this coupling, the curve, $\gamma^{(i)}$, is determined by the GFF $h$ and is called the {\it flow line} of $h$ starting at $X_{i}$.
Therefore, it is ensured that there is a probability law of multiple curves $\left(\gamma^{(i)}:i=1,\dots, N\right)$ that is determined by the GFF. 
Our strategy consists of two steps:
\begin{description}
\item[{\it Step 1}] Show that the random multiple curves $\left(\gamma^{(i)}:i=1,\dots, N\right)$ exhibit the analogous properties as those in Theorem~\ref{thm:main}.
\item[{\it Step 2}] Show that the random geometry generated by $\left(\gamma^{(i)}:i=1,\dots, N\right)$ agrees in probability law with that generated by the $N$-SLE$(\kappa)$ driven by Dyson's Brownian motions $(g_{t}:t\geq 0)$.
\end{description}
As we will see in the rest of the paper, both of these steps are completed as consequences of the characterization of the random multiple curves $\left(\gamma^{(i)}:i=1,\dots, N\right)$ in terms of coupling with the GFF.
We stress that, though the {\it Step 1} is thought of a slight extension of the preceding work \cite{MillerSheffield2016a}, the {\it Step 2} gives a new insight on study of multiple SLE.
Let us describe these steps in more precise language.

\subsubsection*{Step 1}
It should be obvious that the goal of this step is to prove the following statements.
\begin{prop}
\label{prop:properties_flowlines}
We have $\lim_{t\to\infty}|\gamma^{(i)}(t)|=\infty$ a.s. for every $i=1,\dots, N$. Furthermore,
\begin{enumerate}
\item 	if $\kappa\in (0,4]$, $\gamma^{(i)}$, $i=1,\dots, N$ are almost surely disjoint simple curves such that $\gamma^{(i)}(0,\infty)\subset\bH$, $i=1,\dots, N$,
\item 	if $\kappa\in (4,8)$, $\gamma^{(i)}$, $i=1,\dots, N$ almost surely intersect themselves and $\bR$,
\item 	if $\kappa=8$, $\gamma^{(i)}$, $i=1,\dots, N$ are almost surely space-filling.
\end{enumerate}
\end{prop}

We will prove Proposition~\ref{prop:properties_flowlines} in Section~\ref{sect:flowline}. The idea closely follows~\cite{MillerSheffield2016a}, where we will consider conditional laws of a single curve given the others taking advantage of the coordinate-free formulation of the coupling (see Proposition~\ref{prop:coordinate_freeness}).

\subsubsection*{Step 2}
Note that, for each $i=1,\dots, N$, the curve $\gamma^{(i)}$ naturally carries a parametrization $s_{i}\geq 0$ in such a way that the half-plane capacity (see Sect.~\ref{sect:prelim} for definition) of $\gamma^{(i)}(0,s_{i}]$ is $2s_{i}$. Suppose that the curves $\gamma^{(i)}$, $i=1,\dots, N$ are adapted to  filtrations $(\cF^{(i)}_{s_{i}})_{s_{i}\in\bR_{\geq 0}}$, $i=1,\dots, N$, respectively. When we set, for each $N$-tuple $\ul{s}=(s_{1},\dots, s_{N})\in\bR_{\geq 0}^{N}$, the $\sigma$-algebra $\cF_{\ul{s}}=\bigvee_{i=1}^{N}\cF^{(i)}_{s_{i}}$, then the collection $\left(\cF_{\ul{s}}\right)_{\ul{s}\in\bR_{\geq 0}^{N}}$ is a multi-parametric filtration of $\sigma$-algebras with respect to the natural partial order on $\bR_{\geq 0}^{N}$. The next task is to make these parameters depend on a single parameter.

\begin{lem}
\label{lem:single_param}
There exist unique functions $\bR_{\geq 0}\to \bR_{\geq 0};$ $t\mapsto s_{i}(t)$, $i=1,\dots, N$ such that the following property holds.
At each $t\geq 0$, define the subset $\wtilde{K}_{t}\subset\bH$ so that $\bH\bsl \wtilde{K}_{t}$ is the unbounded component of $\bH\bsl \bigcup_{i=1}^{N}\gamma^{(i)}(0,s_{i}(t)]$ and let $\wtilde{g}_{t}:\bH\bsl\wtilde{K}_{t}\to\bH$ be the hydrodynamically normalized conformal map. Then, the family of conformal maps $(\wtilde{g}_{t}:t\geq 0)$ solves the multiple Loewner equation
\begin{equation*}
	\frac{d}{dt}\wtilde{g}_{t}(z)=\sum_{i=1}^{N}\frac{2}{\wtilde{g}_{t}(z)-\wtilde{X}^{(i)}_{t}},\quad t\geq 0, \quad \wtilde{g}_{0}(z)=z,
\end{equation*}
where $\wtilde{X}^{(i)}_{t}=\wtilde{g}_{t}\left(\gamma^{(i)}(s_{i}(t))\right)$, $i=1,\dots, N$.
Furthermore, $s_{i}(t)\to\infty$ as $t\to\infty$, $i=1,\dots, N$.
\end{lem}
Remarkably, the numerators of the summands are normalized as $2$.

The remaining task is to show that the driving processes $\left(\wtilde{X}^{(i)}_{t}:t\geq 0\right)$, $i=1,\dots, N$ give a weak solution to the system of SDEs (\ref{eq:Dyson_timechanged}), which is stated as follows:
\begin{lem}
\label{lem:driving_processes}
At each $t\geq 0$, $\ul{s}(t)=(s_{1}(t),\dots, s_{N}(t))$ is an $\left(\cF_{\ul{s}}\right)_{\ul{s}\in\bR_{\geq 0}^{N}}$-stopping time, and when we set $\cF_{t}=\cF_{\ul{s}(t)}$, $t\geq 0$, the collection $\left(\cF_{t}\right)_{t\geq 0}$ forms a filtration of $\sigma$-subalgebras.
The processes $\left(\wtilde{X}^{(i)}_{t}:t\geq 0\right)$, $i=1,\dots, N$ are adapted to this filtration and satisfy the system of SDEs (\ref{eq:Dyson_timechanged}).
\end{lem}

Lemmas~\ref{lem:single_param} and~\ref{lem:driving_processes}, which will be proved in Sect.~\ref{sect:proofs_lemmas}, prove Theorem~\ref{thm:main}.
\begin{proof}[Proof of Theorem \ref{thm:main}]
Since the probability laws of weak solutions to a system of SDEs coincide,
the probability law of $(g_{t}:t\geq 0)$ agrees with that of $(\wtilde{g}_{t}:t\geq 0)$.
Then, the properties listed in Theorem~\ref{thm:main} are manifest from Proposition \ref{prop:properties_flowlines}.
\end{proof}

We close Introduction with a few comments.
The coupling of the GFF (\ref{eq:relevant_GFF}) and the $N$-SLE$(\kappa)$ was found and discussed in our previous works \cite{KatoriKoshida2020a, KatoriKoshida2020b}, that is, we found that the multiple SLE driven by the time change of $N$ Dyson's Brownian motions of parameter $8/\kappa$ satisfying (\ref{eq:Dyson_timechanged}) is coupled with the GFF (\ref{eq:relevant_GFF}). Lemmas~\ref{lem:single_param} and \ref{lem:driving_processes} above give the converse assertion. In fact, they say that, if a Loewner chain $(\wtilde{g}_{t}:t\geq 0)$ is coupled with the GFF (\ref{eq:relevant_GFF}), then the driving processes must satisfy the system of SDEs (\ref{eq:Dyson_timechanged}). Dyson's Brownian motion model \cite{Dyson1962} is one of the most studied models in dynamical extension of random matrix theory (RMT). We emphasize that Lemmas \ref{lem:single_param} and \ref{lem:driving_processes} are not only essential in the proof of Theorem \ref{thm:main}, but also highlight potential interrelation between GFF and RMT.

In this paper, we focus on the case of $\kappa\in (0,8]$ when the driving processes are almost surely non-colliding. When $\kappa>8$, the driving processes satisfying (\ref{eq:Dyson_timechanged}) are almost surely colliding. In this case, it is natural to expect that a solution can be continued after collision in a suitable manner so that the multiple SLE is coupled with the same GFF. It turned out, however, technically difficult to deal with this colliding regime in the same method relying on GFF. In fact, when $\kappa>8$, the single parametrization in Lemma \ref{lem:single_param} will break down. It is not obvious that this event of breaking down of the single parametrization exactly corresponds to the collision of driving processes. Hence, it seems still challenging to the authors to study property of the $N$-SLE$(\kappa)$ with $\kappa>8$.

We also concentrate on the case when $X_{i}$, $i=1,\dots, N$ in (\ref{eq:relevant_GFF}) are distinct. As we already noted, initial conditions of Dyson's Brownian motions are not necessarily distinct. If two of initial conditions are identical, however, the method we employ in this paper relying on the coupling with GFF does not seem to be applied. We leave analysis of such cases in future work.

We have introduced our driving processes by setting the parameter of Dyson's Brownian motion model at $\beta=8/\kappa$.
In the SLE side, SLE$(\kappa)$ and SLE$(\pr{\kappa})$ with $\pr{\kappa}=16/\kappa$ are known to be {\it dual} to each other~\cite{Dubedat2005,Zhan2008,Dubedat2009b,Zhan2010}.
In terms of the $\beta$-parameter, the two parameters $\beta=8/\kappa$ and $\pr{\beta}=8/\pr{\kappa}$ are related by $\beta=4/\pr{\beta}$.
It is remarkable that, for general $\beta$-ensembles, the static version of Dyson models,
these two parameters are also regarded as dual \cite{Desrosiers2009,Forrester2010}.
It would be an interesting direction to consider if these two kinds of duality can be incorporated in the framework of the coupling
between multiple SLEs and GFFs.

\subsection*{Organization} The rest of the present paper is devoted to proofs of Proposition~\ref{prop:properties_flowlines}, Lemmas \ref{lem:single_param} and \ref{lem:driving_processes}. In Sect.~\ref{sect:prelim}, we give preliminaries on complex analysis and GFF. In Sect.~\ref{sect:flowline}, we give a brief overview of the  coupling of GFFs with Loewner chains followed by a proof of Proposition~\ref{prop:properties_flowlines}. In Sect.~\ref{sect:proofs_lemmas}, we prove Lemmas~\ref{lem:single_param} and \ref{lem:driving_processes} to complete the proof of Theorem~\ref{thm:main}.

\subsection*{Acknowledgments}
We are grateful to Kalle Kyt{\"o}l{\"a} for fruitful discussion and encouragement.
We would also like to thank the anonymous referees for helping us improve our manuscript.
MK was supported by the Grant-in-Aid for Scientific Research (C) (No.~19K03674), (B) (No.~18H01124), (S) (No.~16H06338) and (A) (No.~21H04432) of Japan Society for the Promotion of Science (JSPS).
SK was supported by the Grant-in-Aid for JSPS Fellows (No.~19J01279).

\section{Preliminaries}
\label{sect:prelim}
\subsection{Complex analysis}
A bounded subset $K\subset \bH$ is called a compact $\bH$-hull if $K=\ol{K}\cap \bH$ and $\bH\bsl K$ is simply connected.
For a compact $\bH$-hull $K$, there exists a unique conformal map $g_{K}:\bH\bsl K\to \bH$ under the hydrodynamical normalization: $\lim_{z\to\infty}|g_{K}(z)-z|=0$.
The half-plane capacity of $K$ is defined by $\hcap (K):=\lim_{z\to\infty}z(g_{K}(z)-z)$. Note that the half-plane capacity is positive and increasing in the sense that, if $K\subset \pr{K}$, then $\hcap (K)\le \hcap (\pr{K})$ \cite{Lawler2005}.

Let us recall two significant properties of the half-plane capacity: additivity and scaling properties \cite{Lawler2005}.
Let $A\subset\bH$ be a compact $\bH$-hull and $B\subset\bH\bsl A$ be a subset such that $A\cup B$ be a compact $\bH$-hull.
Then, $g_{A}(B)\subset\bH$ is a compact $\bH$-hull, and we have the additivity of the half-plane capacity:
\begin{equation}
\label{eq:additivity_hcap}
	\hcap (A\cup B)=\hcap (A)+\hcap (g_{A}(B)).
\end{equation}
For a positive number $\alpha>0$, the dilatation $\psi_{\alpha}:\bH\to\bH$ is defined by $\psi_{\alpha}(z)=\alpha z$, $z\in\bH$. 
For a compact $\bH$-hull $A\subset\bH$, its image $\psi_{\alpha}(A)$ under the dilation is again a compact $\bH$-hull.
The scaling property of the half-plane capacity is the property that
\begin{equation}
\label{eq:scaling_hcap}
	\hcap (\psi_{\alpha}(A))=\alpha^{2} \hcap (A).
\end{equation}

Recall that the solution $(g_{t}:t\geq 0)$ of a multiple Loewner equation (\ref{eq:multiple_Loewner}) is a family of conformal maps $g_{t}:\bH\bsl K_{t} \to\bH$, $t\geq 0$. It can be verified from continuity of the driving processes $\left(X^{(i)}_{t}:t\geq 0\right)$, $i=1,\dots, N$ that, at each $t\geq 0$, $K_{t}$ is almost surely a compact $\bH$-hull.
Furthermore, expanding both sides of (\ref{eq:multiple_Loewner}) around $z=\infty$ and comparing the coefficients of $z^{-1}$,
we find that the half-plane capacities $\hcap (K_{t})$, $t\geq 0$ satisfy the ordinary differential equation
\begin{equation*}
	\frac{d}{dt}\hcap (K_{t})=2N,\quad t\geq 0,\quad \hcap (K_{0})=0.
\end{equation*}
In other words, they are given by $\hcap (K_{t})=2Nt$, $t\geq 0$.

\subsection{GFF terminology}
Let $D\subseteq\bC$ be an open subset with harmonically non-trivial boundary, which implies that a Brownian motion starting at a point in $D$ almost surely hits the boundary $\del D$. We write $C^{\infty}_{0}(D)$ for the space of $C^{\infty}$ functions compactly supported in $D$, which we equip with the Dirichlet inner product
\begin{equation*}
	(f,g)_{\nabla}=\frac{1}{2\pi}\int_{D} (\nabla f) \cdot (\nabla g),\ \ f,g\in C^{\infty}_{0}(D).
\end{equation*}
The Hilbert space completion of $C^{\infty}_{0}(D)$ with respect to the inner product $(\cdot,\cdot)_{\nabla}$ is denoted by $W(D)$.
Let us fix a complete orthonormal system  $\set{\phi_{n}}_{n\in\bN}$ of $W(D)$ and take i.i.d. normal Gaussian random variables $\alpha_{n}$, $n\in\bN$ defined on a probability space $(\Omega^{\GFF},\cF^{\GFF},\bP^{\GFF})$. We write the expectation value under the probability measure $\bP^{\GFF}$ as $\bE^{\GFF}$.
The zero-boundary GFF on $D$ is defined by (see e.g.~\cite{MillerSheffield2016a})
\begin{equation*}
	H=\sum_{n\in\bN}\alpha_{n}\phi_{n},
\end{equation*}
which almost surely converges in a space of distributions with test functions in $C_{0}^{\infty}(D)$. Hence, for $f\in C_{0}^{\infty}(D)$, the pairing $(H,f)$ is well-defined, where $(\cdot,\cdot)$ may be regarded as the standard $L^{2}$ inner product. Then, motivated by the integration by parts, we may define
\begin{equation*}
	(H,f)_{\nabla}:=-\frac{1}{2\pi}(H,\Delta f),\ \ f\in C^{\infty}_{0}(D).
\end{equation*}
This extends to the whole $W(D)$ so that $(H,f)_{\nabla}$ is well-defined for $f\in W(D)$. We also have
\begin{equation*}
	\bE^{\GFF}[(H,f)_{\nabla}(H,g)_{\nabla}]=(f,g)_{\nabla},\ \ f,g\in W(D).
\end{equation*}
Hence, this defines an isometry $(H,\cdot)_{\nabla}:W(D)\to L^{2}(\Omega^{\GFF},\cF^{\GFF},\bP^{\GFF})$. We can assume that the $\sigma$-algebra $\cF^{\GFF}$ is generated by $(H,f)_{\nabla}$ $f\in W(D)$, and have the Wiener chaos decomposition \cite{Janson1997, Sheffield2007}:
\begin{equation}
\label{eq:Wiener_chaos_decomposition}
	L^{2}(\Omega^{\GFF},\cF^{\GFF},\bP^{\GFF})\simeq \bigoplus_{n=0}^{\infty}W(D)^{\odot n},
\end{equation}
where $W(D)^{\odot n}$ is the $n$-fold symmetric tensor product of $W(D)$.

Let $U\subsetneq D$ be an open subset. We may regard $W(U)$ as a closed subspace of $W(D)$ and can restrict the GFF $H$ on $D$ to $U$. The orthogonal complement $W^{\perp}(U)\subset W(D)$ consists of functions that are harmonic on $U$.
We define a $\sigma$-subalgebra $\cF^{\GFF}_{U}$ generated by $H|_{U}$, equivalently, generated by $(H,f)$, $f\in C^{\infty}_{0}(U)$. The conditional expectation $\bE^{\GFF}[\cdot|\cF^{\GFF}_{U}]:L^{2}(\Omega^{\GFF},\cF^{\GFF},\bP^{\GFF})\to L^{2}(\Omega^{\GFF},\cF^{\GFF}_{U},\bP^{\GFF})$ is realized in the Wiener chaos decomposition (\ref{eq:Wiener_chaos_decomposition}) by the projection induced by $W(D)\to W(U)$.
For a closed subset $K\subset D$, we write $\cF^{\GFF}_{K^{+}}$ for the $\sigma$-subalgebra generated by the projection of $H$ onto $W^{\perp}(D\bsl K)$.
Then the conditional expectation $\bE^{\GFF}[\cdot|\cF^{\GFF}_{K^{+}}]$ is associated with the projection $W(D)\to W^{\perp}(D\bsl K)$ in the Wiener chaos decomposition (\ref{eq:Wiener_chaos_decomposition}).

According to the orthogonal decomposition $W(D)=W(U)\oplus W^{\perp}(U)$, we write $H\overset{\law}{=}H_{U}+H_{U^{\mathrm{c}}}$, where $H_{U}$ and $H_{U^{\mathrm{c}}}$ are independent random distributions such that (1) $H_{U}$ is the zero-boundary GFF on $U$ and vanishes outside $U$, and (2) $H_{U^{\mathrm{c}}}$ is harmonic on $U$.
Hence, we have the Markov property stated as follows:
\begin{prop}[Markov property]
For a subdomain $U\subset D$, the conditional law of $H|_{U}$ given $\cF^{\GFF}_{(D\bsl U)^{+}}$ is that of the zero-boundary GFF $H_{U}$ plus the harmonic extension of $H|_{\del U}$ onto $U$.
\end{prop}

\section{Flow lines of Gaussian free field}
\label{sect:flowline}
\subsection{Configurations and boundary conditions}
Let $D\subsetneq \bC$ be a simply connected domain, $N\in\bN$, and $x_{1},\dots, x_{N}$ and $y$ be distinct boundary points aligned counterclockwise in this order. We call a set of these data $\cC=(D;x_{1},\dots, x_{N};y)$ a configuration following \cite{Dubedat2009}. An equivalence from a configuration $\cC=(D;x_{1},\dots, x_{N};y)$ to another $\pr{\cC}=(\pr{D};\pr{x}_{1},\dots, \pr{x}_{N};\pr{y})$ is an conformal map $\psi:D\to\pr{D}$ such that $\psi (x_{i})=\pr{x}_{i}$, $i=1,\dots, N$ and $\psi (y)=\pr{y}$. In this case, we informally write $\psi:\cC\to\pr{\cC}$ or $\pr{\cC}=\psi (\cC)$.

Given real numbers $a_{1},\dots, a_{N}$, and $\chi$, we call $(\ul{a},\chi)$ a set of boundary conditions and define, for each configuration $\cC=(D;\ul{x},y)$, a harmonic function $\fh_{\cC,(\ul{a},\chi)}$ on $D$ in the following procedure. For $\cC=(\bH;X_{1},\dots, X_{N};\infty)$, we set
\begin{equation*}
	\fh_{\cC,(\ul{a},\chi)}(z):=\sum_{i=1}^{N}a_{i}\arg (z-X_{i}),\quad z\in\bH.
\end{equation*}
For another configuration $\pr{\cC}=(D;x_{1},\dots, x_{N};y)$ equivalent to $\cC=(\bH;X_{1},\dots, X_{N};\infty)$, we take an equivalence $\psi:\pr{\cC}\to \cC$ and set $\fh_{\pr{\cC},(\ul{a},\chi)}:=\fh_{\cC,(\ul{a},\chi)}\circ \psi-\chi \arg \pr{\psi}$, where $\pr{\psi}$ is the derivative of $\psi$ in $z\in D$. When $N\geq 2$, such an equivalence $\psi$ is unique if exists. Even if it is not unique when $N=1$, the harmonic function $\fh_{\pr{\cC},(\ul{a},\chi)}$ is well-defined.
A pair $([\cC],(\ul{a},\chi))$ of an equivalence class of configurations and a set of boundary conditions is essentially the same object as what is called an imaginary surface in \cite{MillerSheffield2016a}.

\begin{defn}
\label{defn:GFF_bc}
A random distribution $h$ on $D$ with test functions in $C^{\infty}_{0}(D)$ is called a GFF if there exists a configuration $\cC=(D;x_{1},\dots, x_{N};y)$ and a set of boundary conditions $(\ul{a},\chi)$ such that $h\overset{\law}{=}H+\fh_{\cC;(\ul{a},\chi)}$, where $H$ is the zero-boundary GFF on $D$. In this case, we also say that $h$ is the GFF on $\cC$ under the boundary conditions $(\ul{a},\chi)$.
\end{defn}

Since the restriction of a harmonic function to a subdomain is harmonic, the Markov property of a GFF is obvious; for a GFF $h$ on $D$ and any subdomain $U\subset D$, the conditional law of $h|_{U}$ given $\cF^{\GFF}_{(D\bsl U)^{+}}$ is the same as that of the zero-boundary GFF on $U$ plus the harmonic extension of $h|_{\del U}$ onto $U$. It is not, however, ensured that $h|_{U}$ given $\cF^{\GFF}_{(D\bsl U)^{+}}$ is a GFF on $U$ in the sense of Definition \ref{defn:GFF_bc}, i.e., there exists a configuration $\cC=(U;\ul{x};y)$ and a set of boundary conditions $(\ul{a},\chi)$ so that the conditional law of $h|_{U}$ given $\cF^{\GFF}_{(D\bsl U)^{+}}$ agrees with that of $H_{U}+\fh_{\cC,(\ul{a},\chi)}$.

\subsection{Coupling: generality}
The usual notion of a Loewner chain in a domain refers to an increasing closed subsets evolving from the boundary.
Let us introduce its generalization, a {\it Loewner chain in a configuration}, which avoids the other boundary points than the one it starts from.
\begin{defn}
Fix an arbitrary $i\in\set{1,\dots, N}$.
A Loewner chain in a configuration $\cC=(D;x_{1},\dots, x_{N};y)$ starting at $x_{i}$ is a family of increasing subsets $(K^{(i)}_{t}\subset D:t\geq 0)$ such that:
\begin{enumerate}
\item 	it is generated by a continuous non-crossing curve $\gamma^{(i)}:[0,\infty)\to\ol{D}$ starting at $\gamma^{(i)} (0)=x_{i}$ so that, at each $t\geq 0$, $D\bsl K^{(i)}_{t}$ is the connected component of $D\bsl \gamma^{(i)} (0,t]$ whose closure contains $y$.
\item 	it does not hit $x_{j}$, $j\neq i$ and $y$, i.e., $x_{j}\not\in \ol{K^{(i)}_{t}}$, $j\neq i$ and $y\not\in \ol{K^{(i)}_{t}}$ for any $t\geq 0$.
\end{enumerate}
\end{defn}

Let $\cC=(D;x_{1},\dots, x_{N};y)$ be a configuration and fix arbitrarily $i\in \set{1,\dots, N}$.
We write $\cU^{(i)}$ for the collection of open neighborhoods of $x_{i}$ in $D$ equipped with the natural partial order with respect to the inclusion relation.
Let $(\cF^{(i)}_{U})_{U\in\cU^{(i)}}$ be a filtration of $\sigma$-algebras labelled by $\cU^{(i)}$.
An $\cF^{(i)}$-GFF is a GFF $h$ on $D$ such that, for any $U\in\cU^{(i)}$, $\cF^{\GFF}_{U}\subset\cF^{(i)}_{U}$, and $h|_{D\bsl \ol{U}}$ is conditionally independent of $\cF^{(i)}_{U}$ given $\cF^{\GFF}_{\del U}$ (Note that the conditional independence from $\cF^{\GFF}_{U}$ follows from the Markov property). We say that a Loewner chain $(K^{(i)}_{t}:t\ge 0)$ in $\cC$ starting at $x_{i}$ is $\cF^{(i)}$-adapted if, for any $U\in\cU^{(i)}$, its first exit time from $U$ is $\cF^{(i)}_{U}$-measurable. Suppose that the Loewner chain $(K^{(i)}_{t}:t\geq 0)$ is generated by a curve $\gamma^{(i)}$. At any $t\geq 0$, we set $D^{(i)}_{t}:=D\bsl K^{(i)}_{t}$ and $\cC^{(i)}_{t}:=(D^{(i)}_{t};x_{1},\dots, \gamma^{(i)}(t),\dots, x_{N};y)$.
Given a set of boundary conditions $(\ul{a},\chi)$, we consider the following problem referred to as the coupling problem: find a filtered probability space $\left(\Omega,(\cF^{(i)}_{U})_{U\in \cU^{(i)}},\bP\right)$ on which a GFF $h$ and a Loewner chain $(K^{(i)}_{t}:t\ge 0)$ in $\cC$ starting at $x_{i}$ are defined, such that
\begin{itemize}
\item 	the GFF $h$ is an $\cF^{(i)}$-free field,
\item 	the Loewner chain $(K^{(i)}_{t}:t\ge 0)$ is $\cF^{(i)}$-adapted,
\item 	we have
		\begin{equation*}
			\bE\left[h|_{D^{(i)}_{t}}\Big|\cF^{\GFF}_{K_{t}^{(i)+}}\right]=\fh_{\cC_{t}^{(i)},(\ul{a},\chi)}, \quad t\geq 0,
		\end{equation*}
		where $\bE$ denotes the expectation value under the probability measure $\bP$.
\end{itemize}
Recall that the conditional expectation value with respect to $\cF^{\GFF}_{K_{t}^{(i)+}}$ takes the harmonic function part of the restriction $h|_{D_{t}^{(i)}}$.
Hence, the last condition is equivalent to that the conditional law of $h|_{D^{(i)}_{t}}$ given $K^{(i)}_{t}$ agrees with that of the zero boundary GFF on $D^{(i)}_{t}$ plus the harmonic function $\fh_{\cC_{t}^{(i)},(\ul{a},\chi)}$.
Note that, at the initial time $t=0$, we must have $h\overset{\law}{=}H+\fh_{\cC;(\ul{a},\chi)}$, where $H$ is the zero-boundary GFF on $D$, i.e., $h$ is the GFF on $\cC$ under the boundary conditions $(\ul{a},\chi)$.

Significantly, as was noticed in \cite{Dubedat2009}, the above coupling problem is coordinate free in the following sense. 
\begin{prop}
\label{prop:coordinate_freeness}
Let $\pr{\cC}=(D;x_{1},\dots, x_{N};y)$ be a configuration equivalent to $\cC=(\bH;X_{1},\dots,X_{N};\infty)$ under an equivalence $\psi:\pr{\cC}\to \cC$,
and let $i\in\set{1,\dots, N}$ be fixed. 
Suppose that we have a Loewner chain $(K^{(i)}_{t}:t\geq 0)$ in $\cC$ starting at $X_{i}$ and that the pair $\left(h=H+\fh_{\cC,(\ul{a},\chi)},(K^{(i)}_{t}:t\geq 0)\right)$, where $H$ is the zero-boundary GFF on $\bH$, gives a solution to the coupling problem for $\cC$. 
Let $\pr{H}$ be the zero-boundary GFF on $D$ and set $\pr{h}=\pr{H}+\fh_{\pr{\cC},(\ul{a},\chi)}$. 
Then, the pair $\left(\pr{h},\left(\psi^{-1}(K^{(i)}_{t}):t\ge 0\right)\right)$ gives a coupling for $\pr{\cC}$ under the same set of boundary conditions $(\ul{a},\chi)$.
\end{prop}
\begin{proof}
Using the family of conformal maps $(g^{(i)}_{t}:t\geq 0)$ associated with $(K^{(i)}_{t}:t\geq 0)$, we have
\begin{equation*}
	\bE\left[h|_{\bH^{(i)}_{t}}\Big|\cF^{\GFF}_{K_{t}^{(i)+}}\right]=\sum_{j=1}^{N}a_{j}\arg\left(g^{(i)}_{t}(\cdot)-X^{(j)}_{t}\right)-\chi \arg \pr{(g^{(i)}_{t})}(\cdot), \quad t\geq 0,
\end{equation*}
where $X^{(j)}_{t}=g^{(i)}_{t}(X_{j})$, $j=1,\dots, N$.
Note that the composition $g^{(i)}_{t}\circ\psi$ is an equivalence from $\psi^{-1}(\cC^{(i)}_{t})$ to $\cC$, implying that
\begin{align*}
	\fh_{\psi^{-1}(\cC^{(i)}_{t}),(\ul{a},\chi)}&=\left(\sum_{j=1}^{N}a_{j}\arg\left(g^{(i)}_{t}(\cdot)-X^{(j)}_{t}\right)-\chi \arg \pr{(g^{(i)}_{t})}(\cdot)\right)\circ\psi -\chi \arg \pr{\psi}(\cdot) \\
	&=\fh_{\cC^{(i)}_{t},(\ul{a},\chi)}\circ \psi-\chi \arg \pr{\psi}(\cdot), \quad t\geq 0.
\end{align*}
In particular, we have $\fh_{\pr{\cC},(\ul{a},\chi)}=\fh_{\cC,(\ul{a},\chi)}\circ\psi-\chi\arg\pr{\psi}$.
Noticing the conformal invariance of the zero-boundary GFF, we have
\begin{align*}
	\pr{h}|_{\psi^{-1}(\bH^{(i)}_{t})}
	&\overset{\law}{=}\left(H+\fh_{\cC,(\ul{a},\chi)}\right)\Big|_{\bH^{(i)}_{t}}\circ\psi -\chi \arg \pr{\psi} \\
	&\overset{\law}{=}\left(H\circ g^{(i)}_{t}+\fh_{\cC^{(i)}_{t},(\ul{a},\chi)}\right)\circ \psi -\chi \arg \pr{\psi}
\end{align*}
conditioned over $K^{(i)}_{t}$. Since $(H\circ g^{(i)}_{t})\circ \psi$ is the zero-boundary GFF on $\psi^{-1}(\bH^{(i)}_{t})$, we observe that
\begin{align*}
	\bE\left[\pr{h}|_{\psi^{-1}(\bH_{t})}\Big|\cF^{\GFF}_{\psi^{-1}(K^{(i)}_{t})^{+}}\right]=\fh_{\psi^{-1}(\cC^{(i)}_{t}),(\ul{a},\chi)}, \quad t\geq 0.
\end{align*}
Therefore, the desired property holds.
\end{proof}
Due to Proposition \ref{prop:coordinate_freeness}, we can specialize our attention to the case of $\cC=(\bH;X_{1},\dots, X_{N};\infty)$ without any loss of generality as long as a single Loewner chain is concerned.

\subsection{Flow lines}
In the rest of the paper, we consider a specific set of boundary conditions $a_{i}=-\frac{2}{\sqrt{\kappa}}$, $i=1,\dots, N$ and $\chi=\frac{2}{\sqrt{\kappa}}-\frac{\sqrt{\kappa}}{2}$. Observe that the GFF $h$ in (\ref{eq:relevant_GFF}) is the one on $\bH$ under these boundary conditions.
We see that there is a Loewner chain in $\cC=(\bH;X_{1},\dots, X_{N};\infty)$ starting at each $X_{i}$, $i=1,\dots, N$ coupled with this GFF $h$.

Let us arbitrarily fix $i\in\set{1,\dots, N}$, and consider a curve $\gamma^{(i)}:[0,\infty)\to\ol{\bH}$ such that $\gamma^{(i)}(0)=X_{i}$ that is generated by the SLE$(\kappa,\ul{\rho})$ with $\ul{\rho}=(2,\dots, 2)$ as follows \cite{Dubedat2007}. Suppose that $(g^{(i)}_{t}:t\ge 0)$ satisfies the Loewner equation
\begin{equation*}
	\frac{d}{dt}g^{(i)}_{t}(z)=\frac{2}{g^{(i)}_{t}(z)-\xi^{(i)}_{t}},\quad t\ge 0, \quad g^{(i)}_{0}(z)=z\in \bH,
\end{equation*}
where
\begin{align*}
	\xi^{(i)}_{t}&=X_{i}+\sqrt{\kappa}B^{(i)}_{t}+\sum_{j;j\neq i}\int_{0}^{t}\frac{2ds}{\xi^{(i)}_{s}-\zeta^{(i,j)}_{s}},\\
	\zeta^{(i,j)}_{t}&=X_{j}+\int_{0}^{t}\frac{2ds}{\zeta^{(i,j)}_{s}-\xi^{(i)}_{s}}, \quad j\neq i, \quad t\geq 0.
\end{align*}
Here, $(B^{(i)}_{t}:t\ge 0)$ is a standard Brownian motion.
We write $(K^{(i)}_{t}:t\ge 0)$ for the increasing family of compact $\bH$-hulls generated by $(g^{(i)}_{t}:t\ge 0)$, i.e., $\bH\bsl K_{t}^{(i)}$ is the unbounded component of $\bH\bsl \gamma^{(i)}(0,t]$ at each $t\geq 0$.
We also suppose that these processes are adapted to a filtration $(\cF^{(i)}_{t})_{t\ge 0}$.
Then, the family $(K^{(i)}_{t}:t\geq 0)$ is a Loewner chain in $\cC=(\bH;X_{1},\dots, X_{N};\infty)$ starting at $X_{i}$. Indeed, the following property follows from a special case of the arguments in \cite[Appendix A]{KytolaPeltola2016} and the fact that the Bessel process of dimension $d$ almost surely does not hit $0$ if $d\geq 2$ (see e.g. \cite[Theorem 1.1]{Katori2015}):
\begin{prop}[{\cite[Proposition~A.5]{KytolaPeltola2016}}]
\label{prop:absolute_continuity_wrt_Bessel}
Let $i\in\set{1,\dots, N}$ be fixed. Then, for each $j\neq i$, the process $(\xi^{(i)}_{t}-\zeta^{(i,j)}_{t}:t\geq 0)$ is absolutely continuous with respect to a time change of the Bessel process of dimension $8/\kappa+1\geq 2$. In particular, it almost surely does not hit $0$.
\end{prop}
Therefore, the conditions $X_{j}\not\in \ol{K^{(i)}_{t}}$, $j\neq i$, $t\geq 0$ are satisfied.
To describe the coupling, we write $\fh^{(i)}_{t}=\fh_{\cC^{(i)}_{t},(\ul{a},\chi)}$, or explicitly,
\begin{equation*}
	\fh^{(i)}_{t}(\cdot):=-\frac{2}{\sqrt{\kappa}}\arg\left(g^{(i)}_{t}(\cdot)-\xi^{(i)}_{t}\right)-\frac{2}{\sqrt{\kappa}}\sum_{j;j\neq i}\arg\left(g^{(i)}_{t}(\cdot)-\zeta^{(i,j)}_{t}\right)-\chi \arg (g^{(i)}_{t})^{\prime}(\cdot), \quad t\geq 0,
\end{equation*}
where $\chi=\frac{2}{\sqrt{\kappa}}-\frac{\sqrt{\kappa}}{2}$. Then $\fh^{(i)}_{t}$ is a harmonic function on $\bH\bsl K^{(i)}_{t}$ at each $t\geq 0$. We also set $h^{(i)}_{t}:=H\circ g^{(i)}_{t}+\fh^{(i)}_{t}$, $t\geq 0$, where $H$ is the zero-boundary GFF on $\bH$, and write $h=h^{(i)}_{0}$.
The coupling phenomenon is stated as follows~\cite{Dubedat2009, MillerSheffield2016a}:
\begin{prop}
Let $h$ be the GFF in (\ref{eq:relevant_GFF}) and $i\in\set{1,\dots, N}$ be arbitrarily fixed.
For any $t\geq 0$, the conditional law of $h|_{\bH\bsl K^{(i)}_{t}}$ given $\cF^{(i)}_{t}$ is the same as that of $h^{(i)}_{t}$. Moreover, the pair $\left(h,(K^{(i)}_{t}:t\geq 0) \right)$ is a solution to the coupling problem.
\end{prop}
Furthermore, we have the following~\cite{Dubedat2009, MillerSheffield2016a}:
\begin{thm}
\label{thm:strong_coupling}
For each fixed $i\in\set{1,\dots, N}$, the random curve $\gamma^{(i)}$ is a deterministic functional of the GFF $h$ in (\ref{eq:relevant_GFF}) so that $\cF^{(i)}_{t}=\cF^{\GFF}_{K^{(i)+}_{t}}$ at each $t\geq 0$.
\end{thm}

When $\kappa\in (0,4)$, for each $i=1,\dots, N$, the curve $\gamma^{(i)}$ is referred to as the flow line of $h$ starting at $X_{i}$ in \cite{MillerSheffield2016a}. It was shown therein that the curve corresponding to the parameter $\pr{\kappa}=16/\kappa>4$ can be realized in the same domain and called the counter flow line. When $\kappa=4$, each curve $\gamma^{(i)}$, $i=1,\dots, N$ is named a contour line in \cite{SchrammSheffield2013} or a level line in \cite{PeltolaWu2019}. For simplicity, in this paper, we call the curve a flow line regardless of the parameter $\kappa$. Though in \cite{MillerSheffield2016a}, the authors also analyzed the coupling for all $\kappa>0$, we concentrate on the case of $\kappa\in (0,8]$ since, otherwise, the Loewner chain $(K^{(i)}_{t}:t\geq 0)$ cannot be in $\cC=(\bH;X_{1},\dots, X_{N};\infty)$.

Let us state properties of $\gamma^{(i)}$, $i=1,\dots, N$ that are deduced from the above arguments.
%Let us state properties of $\gamma^{(i)}$, $i=1,\dots, N$ that follow from the fact that the SLE$(\kappa,\ul{\rho})$ with $\kappa\in (0,8]$ and $\ul{\rho}=(2,\dots, 2)$ is absolutely continuous with respect to the SLE$(\kappa)$ (e.g. \cite{Werner2004b}) and Proposition~\ref{prop:absolute_continuity_wrt_Bessel}:
\begin{prop}
\label{prop:property_H}
For any $\kappa\in (0,8]$ and $i=1,\dots, N$, $\lim_{t\to\infty}|\gamma^{(i)}(t)|=\infty$. Furthermore,
\begin{enumerate}
\item 	when $\kappa\in (0,4]$, $\gamma^{(i)}$ is almost surely a simple curve such that $\gamma^{(i)}(0,\infty)\subset\bH$,
\item 	when $\kappa\in (4,8)$, $\gamma^{(i)}$ is almost surely self-intersecting and hits the interval $(X_{i-1},X_{i+1})$ (under the agreement that $X_{0}=-\infty$ and $X_{N+1}=\infty$).
\item 	when $\kappa=8$, $\gamma^{(i)}$ is almost surely space-filling away from $(-\infty,X_{i-1}]\cup [X_{i+1},\infty)$ (under the agreement that $(-\infty,X_{0}]=[X_{N+1},\infty)=\emptyset$).
\end{enumerate}
\end{prop}
\begin{proof}
The property that $\lim_{t\to\infty}|\gamma^{(i)}(t)|=\infty$ follows from Proposition~\ref{prop:absolute_continuity_wrt_Bessel} and \cite[Theorem~1.3]{MillerSheffield2016a}.
According to \cite[Remark~2.3]{MillerSheffield2016a} (see also \cite{Werner2004b} for a special case),
our SLE$(\kappa,\ul{\rho})$ with $\kappa\in (0,8]$ and $\ul{\rho}=(2,\dots, 2)$ is absolutely continuous with respect to SLE$(\kappa)$
up to any fixed $t$.
Therefore the remaining properties are consequences of the basic properties of SLE$(\kappa)$ \cite{RohdeSchramm2005}.
\end{proof}

\subsection{Proof of Proposition \ref{prop:properties_flowlines}}
As a consequence of Theorem \ref{thm:strong_coupling}, we have a probability law of an $N$-tuple of random curves $\left(\gamma^{(i)}:i=1,\dots, N\right)$, where, for each $i=1,\dots, N$, $\gamma^{(i)}$ is the flow line of the GFF $h$ in (\ref{eq:relevant_GFF}) starting at $X_{i}$ and is adapted to the filtration $(\cF^{(i)}_{t})_{t\geq 0}$.
We are now at the position of proving Proposition \ref{prop:properties_flowlines}.

\begin{proof}[Proof of Proposition \ref{prop:properties_flowlines}]
Owing to the fact that the coupling problem is coordinate free (Proposition \ref{prop:coordinate_freeness}), we can study a conditional law of $\gamma^{(i)}$ with a fixed $i\in \set{1,\dots, N}$ given segments of other curves $\gamma^{(j)}(0,\sigma_{j}]$ up to stopping times $\sigma_{j}$, $j\neq i$. Let us write $\bH_{\ul{\sigma}}=\bH\bsl\bigcup_{j;j\neq i}K^{(j)}_{\sigma_{j}}$. Then,  it is a consequence of the coupling that the conditional law of $\gamma^{(i)}$ is given by the Loewner chain starting at $X_{i}$ in the configuration
\begin{equation*}
	\cC_{\ul{\sigma}}=\left(\bH_{\ul{\sigma}};\gamma^{(1)}_{\sigma_{1}},\dots, X_{i},\dots, \gamma^{(N)}_{\sigma_{N}};\infty\right),
\end{equation*}
coupled with the GFF under the same set of boundary conditions $(\ul{a},\chi)=(-\frac{2}{\sqrt{\kappa}},\dots, -\frac{2}{\sqrt{\kappa}},\frac{2}{\sqrt{\kappa}}-\frac{\sqrt{\kappa}}{2})$.
Taking a conformal map $\psi:\bH_{\ul{\sigma}}\to\bH$ that fixes $\infty$, it is equivalent to consider the Loewner chain starting at $\psi (X_{i})$ in the configuration
\begin{equation*}
	\psi (\cC_{\ul{\sigma}})=\left(\bH;\psi(\gamma^{(1)}_{\sigma_{1}}),\dots, \psi (X_{i}),\dots, \psi (\gamma^{(N)}_{\sigma_{N}});\infty\right)
\end{equation*}
coupled with the GFF under the same set of boundary conditions.
Let us write the corresponding curve starting at $\psi (X_{i})$ as $\pr{\gamma}$.
Then, Proposition~\ref{prop:property_H} implies the following properties of $\pr{\gamma}$, and equivalently, those of the original $\gamma^{(i)}$ conditioned over the other curves $\gamma^{(j)}(0,\sigma_{j}]$, $j\neq i$:
\begin{enumerate}
\item 	When $\kappa\in (0,4]$, $\pr{\gamma}$ is almost surely a simple curve in $\bH$ and $\pr{\gamma}(0,\infty)\subset\bH$. This means that $\gamma^{(i)}$ conditioned over the segments of the other curves $\gamma^{(j)}(0,\sigma_{j}]$, $j\neq i$ does not hit these segments $\gamma^{(j)}(0,\sigma_{j}]$, $j\neq i$.
\item 	When $\kappa\in (4,8)$, $\pr{\gamma}$ is almost surely self-intersecting and hits the interval $(\psi (\gamma^{(i-1)}_{\sigma_{i-1}}),\psi (\gamma^{(i+1)}_{\sigma_{i+1}}))$. These properties are equivalent to that the original $\gamma^{(i)}$ conditioned over the segments of the other curves $\gamma^{(j)}(0,\sigma_{j}]$, $j\neq i$ is self-intersecting and hits the right boundary of the segment $\gamma^{(i-1)}(0,\sigma_{i-1}]$, the left boundary of the segment $\gamma^{(i+1)}(0,\sigma_{i+1}]$ and the interval $(X_{i-1},X_{i+1})$ under the agreement that $\gamma^{(0)}(0,\sigma_{0}]=\gamma^{(N+1)}(0,\sigma_{N+1}]=\emptyset$, $X_{0}=-\infty$ and $X_{N+1}=\infty$.
\item 	When $\kappa=8$, $\pr{\gamma}$ is space-filling away from $\left(-\infty, \psi (\gamma^{(i-1)}_{\sigma_{i-1}})\right]\cup \left[\psi (\gamma^{(i+1)}_{\sigma_{i+1}}),\infty\right)$, which implies that the original $\gamma^{(i)}$ conditioned over the other curves $\gamma^{(j)}(0,\sigma_{j}]$, $j\neq i$ is space-filling, but does not hit $\gamma^{(i-1)}_{\sigma_{i-1}}$ nor $\gamma^{(i+1)}_{\sigma_{i+1}}$, under the agreement that $\gamma^{(0)}_{\sigma_{0}}=-\infty$ and $\gamma^{(N+1)}_{\sigma_{N+1}}=\infty$.
\end{enumerate}
Since $i=1,\dots, N$ and the stopping times $\sigma_{j}$, $j\neq i$ are arbitrary, Proposition \ref{prop:properties_flowlines} follows.
\end{proof}

\section{Proofs of Lemmas \ref{lem:single_param} and \ref{lem:driving_processes}}
\label{sect:proofs_lemmas}
\subsection{Multiple coupling}

For any $\ul{s}\in\bR_{\geq 0}^{N}$, we define the compact $\bH$-hull $\wtilde{K}_{\ul{s}}\subset \bH$ so that $\bH\bsl \wtilde{K}_{\ul{s}}$ is the unbounded component of $\bH\bsl \bigcup_{i=1}^{N}\gamma^{(i)}(0,s_{i}]$ and write $\wtilde{g}_{\ul{s}}$ for the hydrodynamically normalized conformal map from $\bH\bsl \wtilde{K}_{\ul{s}}$ to $\bH$. At each $\ul{s}\in\bR_{\geq 0}^{N}$, we may consider the following GFF on $\bH\bsl \wtilde{K}_{\ul{s}}$:
\begin{equation*}
	h_{\ul{s}}=H\circ \wtilde{g}_{\ul{s}}-\frac{2}{\sqrt{\kappa}}\sum_{i=1}^{N}\arg (\wtilde{g}_{\ul{s}}(\cdot)-\wtilde{X}^{(i)}_{\ul{s}})-\chi \arg (\wtilde{g}_{\ul{s}})^{\prime}(\cdot),
\end{equation*}
where $\wtilde{X}^{(i)}_{\ul{s}}:=\wtilde{g}_{\ul{s}}(\gamma^{(i)}(s_{i}))$, $i=1,\dots, N$.

\begin{prop}
For any $\ul{s}\in \bR_{\geq 0}^{N}$, the conditional law of $h|_{\bH\bsl K_{\ul{s}}}$ given $\cF_{\ul{s}}$ agrees with that of $h_{\ul{s}}$.
\end{prop}
\begin{proof}
This is again obvious from the fact that the coupling problem is coordinate free (Proposition \ref{prop:coordinate_freeness}).
\end{proof}

\subsection{Proof of Lemma \ref{lem:single_param}}
We define the function $t\mapsto \ul{s}(t)$ as a solution of the system of ordinary differential equations (ODEs)
\begin{equation}
\label{eq:def_single_param}
	\frac{\del \hcap(\wtilde{K}_{\ul{s}})}{\del s_{i}}\frac{ds_{i}(t)}{dt}=2,\quad t\geq 0, \quad i=1,\dots, N,
\end{equation}
under the initial conditions $s_{i}(0)=0$, $i=1,\dots, N$.
Then the desired property is immediately observed from a standard argument in the Loewner theory. In fact, we have
\begin{equation*}
	\frac{\del}{\del s_{i}}\wtilde{g}_{\ul{s}}(z)=\frac{1}{\wtilde{g}_{\ul{s}}(z)-\wtilde{X}^{(i)}_{\ul{s}}}\frac{\del \hcap(\wtilde{K}_{\ul{s}})}{\del s_{i}},
\end{equation*}
and $\wtilde{g}_{t}=\wtilde{g}_{\ul{s}(t)}$, $\wtilde{X}^{(i)}_{t}=\wtilde{X}^{(i)}_{\ul{s}(t)}$, $i=1,\dots, N$, and $\wtilde{K}_{t}=\wtilde{K}_{\ul{s}(t)}$, $t\geq 0$.
The uniqueness of such a function is obvious from the definition. 
Let us see the existence of a solution to (\ref{eq:def_single_param}).
To this aim, we employ an idea from \cite{LawlerSchrammWerner2003} to estimate the partial derivatives $\frac{\del \hcap(\wtilde{K}_{\ul{s}})}{\del s_{i}}$, $i=1,\dots, N$.
Let us take an arbitrary instance of the GFF (\ref{eq:relevant_GFF}). Then, Theorem~\ref{thm:strong_coupling} ensures that instances of $\gamma^{(i)}$, $i=1,\dots, N$ are determined.
We fix $i\in \set{1,\dots, N}$ and $s_{1},\dots, s_{i-1},s_{i+1},\dots, s_{N}\in\bR_{\geq 0}$, and set $A=\wtilde{K}_{(s_{1},\dots, s_{i-1},0,s_{i+1},\dots, s_{N})}$. Then, due to  Proposition~\ref{prop:properties_flowlines}, the curve $\gamma^{(i)}$ lies in $\ol{\bH\bsl A}$. At each $s_{i}\geq 0$, we take the compact $\bH$-hull $K^{(i)}_{A,s_{i}}$ so that $\bH\bsl K^{(i)}_{A,s_{i}}$ is the unbounded component of $\bH\bsl g_{A}(K^{(i)}_{s_{i}})$ and write $g^{(i)}_{A,s_{i}}:=g_{K^{(i)}_{A,s_{i}}}$.
Here, $\left(K^{(i)}_{t}:t\geq 0\right)$ is the family of compact $\bH$-hulls starting at $X_{i}$ that is generated by the instance of $(g^{(i)}_{t}:t\geq 0)$, the SLE$(\kappa,\ul{\rho})$ with $\ul{\rho}=(2,\dots, 2)$ determined by the instance of the GFF. Notice that the image $g_{A}(K^{(i)}_{s_{i}})$, $s_{i}\geq 0$ is not a compact $\bH$-hull if $\gamma^{(i)}(0,s_{i}]$ has hit $A$.
 On the other hand, at each $s_{i}\geq 0$, we define the compact $\bH$-hull $A^{(i)}_{s_{i}}$ so that $\bH\bsl A^{(i)}_{s_{i}}$ is the unbounded component of $\bH\bsl g^{(i)}_{s_{i}}(A)$ and set $f^{(i)}_{s_{i}}:=g_{A^{(i)}_{s_{i}}}$. Again, the image $g^{(i)}_{s_{i}}(A)$, $s_{i}\geq 0$ is not necessarily a compact $\bH$-hull.
Then, we have
\begin{equation}
\label{eq:commutativity_conformal_maps}
	g^{(i)}_{A,s_{i}}\circ g_{A}=f^{(i)}_{s_{i}}\circ g^{(i)}_{s_{i}}, \quad s_{i}\geq 0.
\end{equation}
It follows from the additivity (\ref{eq:additivity_hcap}) of the half-plane capacity that
\begin{equation*}
	\hcap \left(\wtilde{K}_{\ul{s}}\right)=\hcap \left(K^{(i)}_{A,s_{i}}\right)+\hcap (A), \quad \ul{s}=(s_{1},\dots, s_{i-1},s_{i},s_{i+1},\dots, s_{N}), \quad s_{i}\geq 0.
\end{equation*}
Recall that we have assumed that
\begin{equation*}
	\hcap \left(g^{(i)}_{s_{i}}\left(\wtilde{K}_{(s_{1},\dots, s_{i}+\Delta t,\dots, s_{N})}\bsl \wtilde{K}_{\ul{s}}\right)\right)=2\Delta t+ \mathrm{o}(\Delta t), \quad \Delta t \to 0.
\end{equation*}
We can also observe from (\ref{eq:commutativity_conformal_maps}) that
\begin{equation*}
	g^{(i)}_{A,s_{i}}\left(K^{(i)}_{A,s_{i}+\Delta t}\bsl K^{(i)}_{A,s_{i}}\right)=f^{(i)}_{s_{i}}\left(g^{(i)}_{s_{i}}\left(\wtilde{K}_{(s_{1},\dots, s_{i}+\Delta t,\dots, s_{N})}\bsl \wtilde{K}_{\ul{s}}\right)\right),
\end{equation*}
from which and due to the scaling property (\ref{eq:scaling_hcap}) of the half-plane capacity it follows that
\begin{equation*}
	\hcap\left(g^{(i)}_{A,s_{i}}\left(K^{(i)}_{A,s_{i}+\Delta t}\bsl K^{(i)}_{A,s_{i}}\right)\right)=\left((f^{(i)}_{s_{i}})^{\prime}(\xi^{(i)}_{s_{i}})\right)^{2}2\Delta t+\mathrm{o}(\Delta t), \quad \Delta t\to 0.
\end{equation*}
Therefore, we have $\frac{\del}{\del s_{i}}\hcap \left(\wtilde{K}_{\ul{s}}\right)=2\left((f^{(i)}_{s_{i}})^{\prime}(\xi^{(i)}_{s_{i}})\right)^{2}>0$.
This implies that the system of ODEs (\ref{eq:def_single_param}) has a local solution starting at arbitrary initial conditions so that each $s_{i}(t)$, $i=1,\dots, N$ is non-decreasing in $t$.

If the solution starting at $s_{i}(0)=0$, $i=1,\dots, N$ diverges at a finite $T$, i.e., there is an $i\in\set{1,\dots, N}$ such that $s_{i}(t)\to\infty$, as $t\nearrow T$, it contradicts that $\hcap \left(\wtilde{K}_{t}\right)=2Nt$, $t\geq 0$.
It is also seen that $s_{i}(t)\to \infty$ as $t\to\infty$, $i=1,\dots, N$. Indeed, if $s_{i}(t)$ for some $i\in\set{1,\dots, N}$ converges to a finite value as $t\to\infty$, then it means that $\frac{\del \hcap(\wtilde{K}_{\ul{s}})}{ds_{i}}$ diverges.
Hence, Lemma \ref{lem:single_param} has been proved.

\subsection{Proof of Lemma \ref{lem:driving_processes}}
It is obvious from the definition of the function $t\mapsto \ul{s}(t)$ as a solution to the system of differential equations (\ref{eq:def_single_param}) that, for each $t\geq 0$, $\ul{s}(t)$ is an $\left(\cF_{\ul{s}}\right)_{\ul{s}\in\bR_{\geq 0}^{N}}$-stopping time. Then, the $\sigma$-algebra $\cF_{t}:=\cF_{\ul{s}(t)}$, $t\geq 0$ is defined in the standard manner and $(\cF_{t})_{t\ge 0}$ forms a filtration, to which $(\wtilde{X}^{(i)}_{t}:t\ge 0)$, $i=1,\dots, N$ are adapted.
Our goal is to show that the processes $(\wtilde{X}^{(i)}_{t}:t\ge 0)$, $i=1,\dots, N$ give a weak solution to the system of SDEs (\ref{eq:Dyson_timechanged}), which is divided into two parts. Firstly, set $\cC_{t}:=\left(\wtilde{\bH}_{t};\gamma^{(1)}_{s_{1}(t)},\dots, \gamma^{(N)}_{s_{N}(t)};\infty\right)$, $t\geq 0$ and consider the following stochastic process
\begin{equation*}
	\fh_{t}(\cdot):=\fh_{\cC_{t};(\ul{a},\chi)}(\cdot)=-\frac{2}{\sqrt{\kappa}}\sum_{i=1}^{N}\arg (\wtilde{g}_{t}(\cdot)-\wtilde{X}^{(i)}_{t})-\chi \arg(\wtilde{g}_{t})^{\prime}(\cdot),\quad t\ge 0.
\end{equation*}
The coupling condition implies the following property:
\begin{prop}
\label{prop:localmartingale}
The stochastic process $(\fh_{t}(\cdot):t\ge 0)$ is a family of continuous local martingales with cross variation
\begin{equation*}
	d\braket{\fh(z),\fh(w)}_{t}=-dG(\wtilde{g}_{t}(z),\wtilde{g}_{t}(w)),\ \ z,w\in \bH,\ \ 0\leq t\leq \tau_{z}\wedge\tau_{w},
\end{equation*}
where $G(z, w):=\log \big|\frac{z-\ol{w}}{z-w}\big|$ is the Green's function on $\bH$.
\end{prop}
\begin{proof}
From the coupling condition, the conditional law of $h|_{\bH\bsl K_{t}}$ given $\cF_{t}$ agrees with that of $H\circ g_{t}+\fh_{t}$.
Let $f\in C^{\infty}_{0}(D)$ be a test function and $\varepsilon>0$ be a constant.
Setting $\tau_{f,\varepsilon}=\inf\set{t>0:\dist(K_{t},\supp (f))\le \varepsilon}$, we can see that $M_{t}:=\bE\left[(h,f)|\cF_{t\wedge\tau_{f,\varepsilon}}\right]$, $t\ge 0$ is a martingale. 
Recall that, from Theorem \ref{thm:strong_coupling}, we have $\cF_{t}=\bigvee_{i=1}^{N}\cF^{(i)}_{s_{i}(t)}=\bigvee_{i=1}^{N}\cF^{\GFF}_{K_{s_{i}(t)}^{(i)+}}$, $t\geq 0$, which implies that $M_{t}=(\fh_{t\wedge \tau_{f,\varepsilon}},f)$, $t\geq 0$. Hence, $(\fh_{t}:t\ge 0)$ is a local martingale.
For two test functions $f_{1}$ and $f_{2}$ in $\bH\bsl K_{t}$ we set
\begin{equation*}
	E_{t}(f_{1},f_{2}):=\int_{\supp (f_{1})\times \supp (f_{2})}f_{1}(z)G(\wtilde{g}_{t}(z),\wtilde{g}_{t}(w))f_{2}(w)dzdw, \quad 0\leq t\leq \tau_{f_{1},\varepsilon}\wedge \tau_{f_{2},\varepsilon}.
\end{equation*}
Note that $(h,f_{i})$, $i=1,2$ are Gaussian variables with mean $(\fh_{0},f_{i})$, $i=1,2$, respectively and their covariance is $E_{0}(f_{1},f_{2})$. On the other hand, $(H\circ \wtilde{g}_{t},f_{i})$, $i=1,2$ are Gaussian variables with mean $0$ and their covariance is $E_{t}(f_{1}, f_{2})$. Hence, $(\fh_{t},f_{i})$, $i=1,2$ must be Gaussian variables with mean $(\fh_{0},f_{i})$, $i=1,2$, respectively, and covariance $E_{0}(f_{1},f_{2})-E_{t}(f_{1},f_{2})$. Therefore, the cross variation of $(\fh_{t}:t\ge 0)$ must be as desired.
\end{proof}

Therefore, the proof of Lemma \ref{lem:driving_processes} is reduced to proving the following result:
\begin{prop}
Suppose that a Loewner chain $(g_{t}:t\ge 0)$ solves the multiple Loewner equation (\ref{eq:multiple_Loewner}).
Consider the stochastic process
\begin{equation*}
	\fh_{t}(\cdot):=-\frac{2}{\sqrt{\kappa}}\sum_{i=1}^{N}\arg (g_{t}(\cdot) -X^{(i)}_{t})-\chi \arg \pr{g}_{t}(\cdot),\ \ t\ge 0.
\end{equation*}
Then $(\fh_{t}(\cdot):t\ge 0)$ is a family continuous local martingales with cross variation
\begin{equation*}
	d\braket{\fh(z),\fh(w)}_{t}=-d G(g_{t}(z),g_{t}(w)),\quad z,w\in \bH,\quad  0\leq t\leq \tau_{z}\wedge \tau_{w},
\end{equation*}
if and only if the driving processes $(X^{(i)}_{t}:t\ge 0)$, $i=1,\dots, N$ solve the system of SDEs (\ref{eq:Dyson_timechanged}).
\end{prop}

\begin{proof}
It can be argued that $(\fh_{t}(\cdot):t\ge 0)$ is a continuous local martingale if and only if so is
\begin{equation*}
	\ol{\fh}_{t}(\cdot):=-\frac{2}{\sqrt{\kappa}}\sum_{i=1}^{N}\log (g_{t}(\cdot) -X_{t}^{(i)})-\chi \log \pr{g}_{t}(\cdot),\ \ t\ge 0.
\end{equation*}
Note that $\fh_{t}(\cdot)=\rmIm \ol{\fh}_{t}(\cdot)$, $t\ge 0$.
From the multiple Loewner equation (\ref{eq:multiple_Loewner}), we have
\begin{equation*}
	\frac{d}{dt}\log \pr{g}_{t}(z)=-\sum_{i=1}^{N}\frac{2}{(g_{t}(z)-X^{(i)}_{t})^{2}}, \quad z\in \bH \quad 0\leq t\leq \tau_{z}.
\end{equation*}
Let us take generic points $z_{1},\dots, z_{N}\in\bH$ and set $\tau_{\ul{z}}=\tau_{z_{1}}\wedge\cdots\wedge \tau_{z_{N}}$.
Then, we know that $\ol{\fh}_{t}(z_{j})$, $j=1,\dots, N$ are continuous local martingales and $g_{t}(z_{j})$, $\log \pr{g}_{t}(z_{j})$, $j=1,\dots, N$ are finite variational processes up to $\tau_{\ul{z}}$.
Owing to the implicit function theorem, we can see that each $X^{(i)}_{t}$, $i=1,\dots, N$ is a $C^{\infty}$ function of $\ol{\fh}_{t}(z_{j})$, $g_{t}(z_{j})$, $\log \pr{g}_{t}(z_{j})$, $j=1,\dots, N$ up to $\tau_{\ul{z}}$. Therefore, it is a semi-martingale due to It{\^o}'s theorem.
We write $X^{(i)}_{t}=M^{(i)}_{t}+F^{(i)}_{t}$, $t\ge 0$, $i=1,\dots, N$, where $(M^{(i)}_{t}:t\ge 0)$, $i=1,\dots, N$ are local martingales and $(F^{(i)}_{t}:t\ge 0)$, $i=1,\dots, N$ are finite variation processes. Then, we compute the stochastic derivative of $(\ol{\fh}_{t}(\cdot):t\ge 0)$ to see
\begin{align*}
	d\ol{\fh}_{t}(z)=&\sum_{i=1}^{N}\frac{1}{(g_{t}(z)-X^{(i)}_{t})^{2}}\left(\left(-\frac{4}{\sqrt{\kappa}}+2\chi\right)dt+\frac{1}{\sqrt{\kappa}}d\braket{M^{(i)},M^{(i)}}_{t}\right) \\
	&+\frac{2}{\sqrt{\kappa}}\sum_{i=1}^{N}\frac{1}{g_{t}(z)-X^{(i)}_{t}}\Biggl(dF^{(i)}_{t}-\sum_{\substack{j=1 \\ j\neq i}}^{N}\frac{4}{X^{(i)}_{t}-X^{(j)}_{t}}dt\Biggr) \\
	&+\frac{2}{\sqrt{\kappa}}\sum_{i=1}^{N}\frac{1}{g_{t}(z)-X^{(i)}_{t}}dM^{(i)}_{t}, \quad z\in \bH, \quad 0\leq t\leq \tau_{z}.
\end{align*}
Here we used the identity
\begin{equation*}
	\frac{1}{2}\sum_{i=1}^{N}\frac{1}{z-x_{i}}\sum_{\substack{j=1\\ j\neq i}}^{N}\frac{1}{z-x_{j}}=\sum_{i=1}^{N}\frac{1}{z-x_{i}}\sum_{\substack{j=1\\ j\neq i}}^{N}\frac{1}{x_{i}-x_{j}}
\end{equation*}
between rational functions, which can be shown either by a direct manipulation or comparing the residues at the poles at $z=x_{i}$, $i=1,\dots, N$.

The contribution from the pole of second order at $g_{t}(z)=X^{(i)}_{t}$, $i=1,\dots, N$ forces the quadratic variation of the local martingale to be $d\braket{M^{(i)},M^{(i)}}_{t}=\kappa dt$, $i=1,\dots, N$. Looking at the pole of first order at $g_{t}(z)=X^{(i)}_{t}$, $i=1,\dots, N$, we conclude that $dF^{(i)}_{t}=\sum_{j=1;\;j\neq i}^{N}\frac{4}{X^{(i)}_{t}-X^{(j)}_{t}}dt$, $t\ge 0$, $i=1,\dots, N$.
Furthermore, the cross variation of $(\fh_{t}(\cdot):t\ge 0)$ is computed as
\begin{align*}
	d\braket{\fh(z),\fh(w)}_{t}
	=&-dG(g_{t}(z),g_{t}(w)) \\
	&+\frac{4}{\kappa}\sum_{\substack{i,j=1 \\i\neq j}}^{N}\rmIm\left(\frac{1}{g_{t}(z)-X^{(i)}_{t}}\right)\rmIm\left(\frac{1}{g_{t}(w)-X^{(j)}_{t}}\right)d\braket{M^{(i)},M^{(j)}}_{t}, \\
	&z,w\in \bH, \quad 0\leq t\leq \tau_{z}\wedge \tau_{w}.
\end{align*}
Hence, we have $d\braket{M^{(i)},M^{(j)}}_{t}=0$, $t\ge 0$, $i\neq j$.
\end{proof}

\bibliographystyle{alpha}
\bibliography{sle_gff}

\end{document}